\documentclass[a4paper,12pt,reqno]{amsart}

\usepackage[utf8]{inputenc}
\usepackage{amssymb}
\usepackage{amsmath}
\usepackage{amsthm}
\usepackage{amsfonts}
\usepackage{mathtools}
\usepackage{hyperref}
\usepackage{cleveref}
\usepackage[pdftex,dvipsnames]{xcolor}
\usepackage{todonotes}
\usepackage{comment}
\usepackage{mathrsfs}
\usepackage{enumitem}
\usepackage[normalem]{ulem}
\usepackage{float}

\makeatletter
\def\dicho#1{\expandafter\@dicho\csname c@#1\endcsname}
\def\@dicho#1{\ifnum#1>1 or \else\fi(\@Roman#1)}
\makeatother
\AddEnumerateCounter{\dicho}{\@dicho}{or (III)}
\newlist{dichotomy}{enumerate}{1}
\setlist[dichotomy]{label=\dicho*,leftmargin=1.5cm}

\usepackage[british]{babel}
\usepackage[margin=1.2in]{geometry}
\usepackage{eqparbox}

\usepackage{pbox}

\renewcommand{\O}{\mathcal{O}} 

\newcommand{\ooverline}[1]{\overline{#1\rule{0pt}{6pt}}}

\newcommand{\set}[1]{\left\lbrace #1 \right\rbrace}

\newcommand{\diamondop}[1]{\langle #1 \rangle}
\newcommand{\field}[1]{\mathbb{#1}}  
\newcommand{\Q}{\field{Q}} 
\newcommand{\C}{\field{C}} 
\newcommand{\Z}{\field{Z}} 
\newcommand{\F}{\field{F}} 
\newcommand{\A}{\field{A}}

\newcommand{\Fbar}{\ooverline{\F}}
\renewcommand{\P}{\field{P}}
\newcommand{\PP}{\field{P}}

\DeclareMathOperator{\Char}{char}

\DeclareMathOperator{\Aut}{Aut}
\DeclareMathOperator{\PSL}{PSL}
\DeclareMathOperator{\SL}{SL}

\DeclareMathOperator{\Pic}{Pic}

\DeclareMathOperator{\Sym}{Sym}

\DeclareMathOperator{\gon}{gon}
\DeclareMathOperator{\aut}{Aut}

\DeclareMathOperator{\Proj}{Proj}

\DeclareMathOperator{\rank}{rank}

\newtheorem{lemma}{Lemma}
\newtheorem{theorem}[lemma]{Theorem}
\newtheorem{proposition}[lemma]{Proposition}
\newtheorem{corollary}[lemma]{Corollary}
\newtheorem{conjecture}[lemma]{Conjecture}

\theoremstyle{definition}
\newtheorem{definition}[lemma]{Definition}

\newtheorem{question}[lemma]{Question}
\newtheorem{remark}[lemma]{Remark}

\numberwithin{lemma}{section}
\numberwithin{equation}{section}
\numberwithin{figure}{section}

\newcommand{\filip}[1]{{\color{purple} \sf $\diamondsuit\diamondsuit\diamondsuit$ Filip: [#1]}}
\newcommand{\maarten}[1]{{\color{Green} \sf $\diamondsuit\diamondsuit\diamondsuit$ Maarten: [#1]}}

\title{Hyperelliptic and trigonal modular curves in characteristic $p$}
\author[Derickx]{Maarten Derickx}
\address{Maarten Derickx, Den Haag, The Netherlands}
\email{\url{maarten@mderickx.nl}}
\urladdr{\url{http://www.maartenderickx.nl/}}

\author[Najman]{Filip Najman}
\address{Filip Najman, University of Zagreb, Bijeni\v{c}ka Cesta 30, 10000 Zagreb, Croatia}
\email{\url{fnajman@math.hr}}
\urladdr{\url{https://web.math.pmf.unizg.hr/~fnajman/}}
\thanks{F.N. was supported by the  project “Implementation of cutting-edge research and its application as part of the Scientific Center of Excellence for Quantum and Complex Systems, and Representations of Lie Algebras“, PK.1.1.02, European Union, European Regional Development Fund. and by the Croatian Science Foundation under the project no. IP-2022-10-5008.}

\date{}

\makeatletter
\providecommand\@dotsep{5}
\renewcommand{\listoftodos}[1][\@todonotes@todolistname]{%
  \@starttoc{tdo}{#1}}
\makeatother

\begin{document}

\maketitle

\begin{abstract} Let $X_\Delta(N)$ be an intermediate modular curve of level $N$, meaning that there exist (possibly trivial) morphisms $X_1(N)\rightarrow X_\Delta(N) \rightarrow X_0(N)$.
For all such intermediate modular curves, we give an explicit description of all primes $p \nmid N$ such that $X_\Delta(N)_{\overline{\mathbb F}_p}$ is either hyperelliptic or trigonal. Furthermore we also determine all primes $p$ such that $X_\Delta(N)_{\F_p}$ is trigonal.

This is done by first using the Castelnuovo-Severi inequality to establish a bound $N_0$ such that if $X_0(N)_{{\overline{\mathbb F}_p}}$ is hyperelliptic or trigonal, then $N \leq N_0$. To deal with the remaining small values of $N$, we develop a method based on the careful study of the canonical ideal to determine, for a fixed curve $X_\Delta(N)$, all the primes $p$ such that the $X_\Delta(N)_{ {\overline{\mathbb F}_p}}$ is trigonal or hyperelliptic.

Furthermore, using similar methods, we show that $X_\Delta(N)_{{\overline{\mathbb F}_p}}$ is not a smooth plane quintic, for any $N$ and any $p$.
\end{abstract}

\section{Introduction}
Let $k$ be a field and $C$ a curve over $k$. Throughout the paper we will assume all curves are geometrically integral and proper over $k$. The \textit{gonality} $\gon_k C$ of $C$ over $k$ is the least degree of a non-constant morphism $f:C\rightarrow \PP^1_k$. Equivalently, it is the least degree of a non-constant function $f\in k(C)$.

We will study the gonality of modular curves, and in particular of ``intermediate" modular curves $X_\Delta(N)$. Let $N$ be a positive integer and let $\Delta$ be a subgroup of $(\Z/ N\Z)^\times /\diamondop{-1}$; we will say that $\Delta$ is of level $N$. Let $X_\Delta(N)$ be the modular curve, defined over $\Q$, associated to the modular group $\Gamma_\Delta=\Gamma_\Delta(N)$ defined by
$$\Gamma_\Delta=\left\{
\begin{pmatrix} a & b\\ c & d\end{pmatrix}\in \SL_2(\Z) \mid c \equiv 0 \text { mod } N, (a \text{ mod } N) \in \Delta
\right\}.$$
If $\Delta$ is the trivial subgroup of $(\Z/ N\Z)^\times /\diamondop{-1}$, then $X_\Delta(N)=X_1(N)$ and for $\Delta=(\Z/ N\Z)^\times /\diamondop{-1}$ we have $X_\Delta(N)=X_0(N)$. The morphisms $X_1(N)\rightarrow X_0(N)$ factor through $X_\Delta(N)$:  $X_1(N)\rightarrow X_\Delta(N) \rightarrow X_0(N)$.

Gonalities of modular curves over fields of characteristic $0$, most commonly $\Q$ or $\C$, have received considerable attention. The modular curves that have been most studied are $X_0(N)$ and $X_1(N)$. For $X_0(N)$ the results are as follows. Ogg \cite{Ogg74} determined the hyperelliptic modular curves $X_0(N)$. Hasegawa and Shimura \cite{HasegawaShimura_trig} determined the $X_0(N)$ that are trigonal over $\C$ and over $\Q$, and Jeon and Park \cite{JeonPark05} determined the $X_0(N)$ that are tetragonal over $\C$. More recently, Najman and Orlić \cite{NajmanOrlić} determined the $X_0(N)$ that are tetragonal and pentagonal over $\Q$ and determined the gonality of $X_0(N)$ for all $N<135.$

All the curves $X_1(N)$ with $\gon_\Q X_1(N)=d$ were determined for $d=2$ by Ishii and Momose \cite{IshiiMomose} \footnote{Ishii and Momose also determined all the hyerelliptic $X_\Delta(N)$, though there is an error in the paper; see \cite[Lemma 4.2]{JeonKim05}. In \cite{JeonKim2020}[Thm 3.6] it was finally shown that over $\C$ the curve $X_\Delta(21)$ with $\Delta = \set{\pm 1, \pm 8}$ is indeed the only hyperelliptic intermediate modular curve $X_\Delta(N)$ that is not of the form $X_0(N)$ or $X_1(N)$. Note that our results do not rely on \cite{IshiiMomose}.}, for $d=3$ by Jeon, Kim and Schweizer \cite{JeonKimSchweizer04} and for $d=4$ by Jeon, Kim and Park \cite{JeonKimPark06} and for $5 \leq d \leq 8$ by Derickx and van Hoeij \cite{derickxVH}. Derickx and van Hoeij \cite{derickxVH} also determined $\gon_\Q X_1(N)$ for all $N\leq 40$ and gave upper bounds for $N\leq 250$.

More generally, Abramovich \cite{abramovich} gave a lower bound for the gonality of any modular curve over $\C$ (which is usually not sharp). From this result, it easily follows that there are only finitely many modular curves $X$ with $\gon_\C X \leq d$ for some fixed positive integer $d$.

 In this paper we consider the gonality of the modular curves $X_\Delta(N)$ over $\F_p$ and $ {\Fbar_p}$ for primes $p$ of good reduction. 
 Poonen \cite{Poonen2007} showed that if one fixes the prime $p$, then the set of $\Gamma$ such that $\gon_{{\Fbar_p}}{X_\Gamma}\leq d$ is finite, and gave lower bounds on $\gon_{\F_{p^2}}{X_\Gamma}$ and $\gon_{\Fbar_{p}}{X_\Gamma}$, depending on $p$ and the index of $\Gamma$ in $\PSL_2(\Z)$ \cite[Proposition 3.1]{Poonen2007}.

 In \Cref{t1} and \Cref{t2} we will explicitly find all the pairs $(\Gamma_\Delta,p)$, with $p$ not dividing the level of $\Gamma_\Delta$ such that $\gon_{\F_p}{X_\Gamma}=2$ or $\gon_{\F_p}{X_\Gamma}=3$.

 It is easy to see that, for any curve $X$, $\gon_{\F_{p}}X\leq \gon_{\Q}X$. It is natural to ask, in the case of (some sets of) modular curves, when does equality hold? A question in this direction, which was also one of the motivations for this paper, was asked by David Zureick-Brown on MathOverflow\footnote{\tiny \url{https://mathoverflow.net/questions/132618/hyperelliptic-modular-curves-in-characteristic-p}}.

\begin{question} \label{q1}
Are there any $N$ such that $X_0(N)_\Q$ is not hyperelliptic but for some $p$ not dividing $N$, $X_0(N)_{\F_p}$ is hyperelliptic?
\end{question}

We show in \Cref{t1} below that the answer to this question is negative. We consider the following more general questions.

\begin{question} \label{q2} \,

\begin{itemize}
    \item[A)] Given some family of modular curves $S$ and a positive integer $d$, can we determine all $X\in S$ and primes $p$ of good reduction for $X$ such that
$$\gon_{\F_p}X = d?$$
\item[B)]
Given $S$ and $d$, can we also determine the $X$ and $p$ as above such that
$$\gon_{\Fbar_p}X = d?$$
\end{itemize}
\end{question}

For $d=2$ both versions of the question are equivalent for all $X \in S$, since  over a finite field $\F_p$ every conic has a point, hence $X$ is hyperelliptic over $\F_p$ if and only if it is hyperelliptic over $ \Fbar_p$. We answer these questions for $d=2$ and $3$ where $S$ is the set of all intermediate modular curves $X_\Delta(N)$.

\begin{theorem} \label{t1}
Let $N$ be a positive integer, $p$ a prime not dividing $N$, and $X_\Delta(N)$ an intermediate modular curve of level $N$. Then
$$2=\gon_{{\Fbar}_p}X_\Delta(N) < \gon_{\C}X_\Delta(N).$$
if and only if $N=37$, $\Delta=\langle 4 \rangle \leq (\Z/ N\Z)^\times$ and $p=2$.
\end{theorem}


\begin{theorem} \label{t2}
Let $N$ be a positive integer, $p$ a prime not dividing $N$, and $X_\Delta(N)$ an intermediate modular curve of level $N$. Then
$$3=\gon_{{\Fbar}_p}X_\Delta(N) < \gon_{\C}X_\Delta(N).$$
if and only if $X_\Delta(N)=X_0(73)$ and $p=2$.
\end{theorem}

To answer \Cref{q2} A) for $d=3$, in \Cref{sec:fields_of_definition} we determine the fields of definition of all trigonal maps $X_\Delta(N) \rightarrow \PP^1$ in characteristic $p$, for all $p$ not dividing $N$.

Our results give evidence to the following conjecture.
\begin{conjecture}
    For any fixed integer $d$, there exist finitely many modular curves $X_\Gamma$ such that $(X_\Gamma)_{\overline \F_p}$ is of gonality $d$ for some prime $p$ of good reduction.
\end{conjecture}
\noindent Our results prove this conjecture for $d=2$ and $3$ and $X_\Gamma$ of the form $X_\Delta(N)$, as it is known that there are fintely many modular curves of fixed gonality over $\Q$ (see \cite{BGGP05}).

Poonen \cite[p.692]{Poonen2007} already mentioned that it is likely that a stronger version of the above conjecture is true. Indeed, one could expect that there is a constant $c > 0$ such that for all congruence subgroups $\Gamma \subseteq \PSL_2(\Z)$ and all primes not dividing the level of $\Gamma$ that $\gon_{\overline \F_p} (X_\Gamma) \geq c \, [\PSL_2(\Z): \Gamma]$.

Anni, Assaf and Lorenzo Garc\'ia recently proved \cite{AALG}, among other results, that there are no modular curves $X_\Delta(N)$ that have a smooth plane model of degree $5$. We prove the same result for all $X_\Delta(N)$ over all ${\Fbar_p}$.

\begin{theorem}\label{t3}
Let $N$ be a positive integer, $p$ a prime not dividing $N$, and $X_\Delta(N)$ an intermediate modular curve of level $N$. Then $X_{\Delta}(N)_{{\Fbar_p}}$ is not a smooth plane quintic.

\end{theorem}

We prove our results in 2 steps. First, we show that the set of $N$ we need to consider is such that the class number $h(-4N)$ is not too large (see \Cref{lem:ramification_formula} and \Cref{prop:conditions}). The (now proven) Gauss conjecture then reduces the problem to dealing with only finitely many $N$. This strategy could in principle also be used to classify all hyperelliptic and trigonal modular curves over $\mathbb Q$ and $\mathbb C$.

Second, to deal with the remaining $N$, we develop explicit computational criteria (see \Cref{sec:hyp_small}), based on Petri's theorem, to check whether, for a given $N$, there exists an $X_\Delta(N)$ of level $N$ and a prime $p$ not dividing $N$ such that $X_\Delta(N)_{\F_p}$ is hyperelliptic/trigonal/a smooth plane quintic.

All the code used to obtain our results can be found at \cite{code}.

\subsection{Acknowledgements}
We thank John Voight and David Zureick-Brown for their helpful discussion and pointers to relevant literature. We are also thankful to Petar Orlić for pointing out some mistakes in an earlier version of this paper.

\section{Background and notation}
\label{sec:background}

We now set up notation that will be used throughout the paper. For a field $k$, we denote the separable closure of $k$ by $k^{sep}$. By $X_0(N)$ we denote the classical modular curve parametrizing isomorphism classes of pairs $(E, C)$ of generalized elliptic curves together with a cyclic subgroup $C$ of order $N$. We will call a divisor $d>1$  of $N$ an \textit{Atkin-Lehner divisor} if $(d,N/d)=1$. For any Atkin-Lehner divisor $d$ of $N$, with $N=dm$, the \textit{Atkin-Lehner involution} $w_d$ acts on a point $x\in Y_0(n)$, where $x$ corresponds to $(E,C_d,C_{m})$ with $C_d$ and $C_m$ being cyclic groups of order $d$ and $m$, to the point $w_d(x)$, corresponding to $(E/C_d,E[d]/C_d,(C_{m}+C_d)/C_d)$. The Atkin-Lehner involutions form a subgroup of $\Aut(X_0(N))$ isomorphic to $(\Z/2\Z)^{\omega(N)}$, where $\omega(N)$ is the number of prime divisors of $N$. The curve $X_0(N)$ and all its Atkin-Lehner involutions are defined over $\Q$. The quotient $X_0(N)/w_N$ is denoted by $X_0^+(N)$ and the quotient of $X_0(N)$ by the whole group of Atkin-Lehner involutions is denoted by $X_0^*(N)$.

The number of the ramification points $\nu(d, N)$ of the map $X_0(N)\rightarrow X_0(N)/w_d$ will be given in terms of class numbers of imaginary quadratic fields (see \Cref{lem:ramification_formula} and \Cref{prop:conditions}). We give the relevant data, which is taken from \cite{Watkins2004}, about quadratic imaginary fields of class number up to $100$ in \Cref{appendix}

By $X_1(N)$ we denote the modular curve whose $k$-rational points parameterize pairs $(E,P)$ where $E$ is a generalized elliptic curve and $P\in E(k)$ is a point $N$, up to $k$-isomorphism. For a $d\in (\Z/N\Z)^\times /\diamondop{-1}$, the diamond operator $\langle d\rangle $ acts on $X_1(N)$ by sending $(E,P)$ to $(E,dP)$. This makes it possible to see $ (\Z/N\Z)^\times /\diamondop{-1}$ as a subgroup of $\Aut X_1(N)$. For a subgroup $\Delta \leq (\Z/N\Z)^\times /\diamondop{-1}$, $X_\Delta(N)$ is the quotient of $X_1(N)$ by the automorphism group of diamond operators $\langle d \rangle$, for $d\in \Delta$. We note that for $\Delta=(\Z/N\Z)^\times$, we have $X_\Delta(N) =X_0(N)$ and for $\Delta=\diamondop{-1}$, we have $X_\Delta(N)=X_1(N)$. For any $\Delta \leq (\Z/N\Z)^\times /\diamondop{-1}$, the curve $X_\Delta(N)$ lies "in between" $X_0(N)$ and $X_1(N)$, in the sense that there are (Galois) morphisms
$$X_1(N)\rightarrow X_\Delta(N)\rightarrow X_0(N);$$
hence the curves $X_\Delta(N)$ are called \textit{intermediate modular curves}.

Let $d$ be an Atkin-Lehner divisor of $N$, $m=N/d$ and $\zeta_d$ an d-th root of unity. Note that $w_d$ can be lifted to an automorphism $w_{d,\zeta_d}$  of $X_1(N)$ over $\Z[1/N][\zeta_d]$. On a point $x\in Y_1(N)$ this automorphism has the following description. Let $(E,P_d,P_m)$ be an elliptic curve with points $P_d,P_m$ of order $d$ and $m$, so that $(E,P_d+P_m)$ corresponds to $x$. Then $w_{d,\zeta_d}(x)$ is the point corresponding to $(E/\langle P_d\rangle, Q_d + (P_m \mod \langle P_d\rangle ) )$ where $Q_d$ is the point in $E/\langle P_d\rangle$ such that any preimage $Q'_d$ in $E$ satisfies $\langle Q'_d, P_d \rangle = \zeta_d$, where $\langle \_,\_ \rangle$ is the Weil pairing.  Note if $l$ is an integer co-prime to $d$ then $w_{d,\zeta_d^l}(x)$ corresponds to $(E/\langle P_d\rangle, \langle l \rangle Q_d + (P_m \mod \langle P_d\rangle ) )$. It follows that for $l \equiv 1 \mod m$ we have $w_{d,\zeta_d^l} = \langle l \rangle w_{d,\zeta_d}$.

\begin{remark}
	Let $\Delta \leq (\Z/N\Z)^\times /\diamondop{-1}$ be a subgroup; then in general $w_{d,\zeta_d}$ doesn't induce an automorphism of $X_\Delta(N)_{\Z[1/N][\zeta_d]}$. However if $w_{d,\zeta_d} \Delta (w_{d,\zeta_d})^{-1} = \Delta$ then $w_{d,\zeta_d}$ does induce an automorphism of $X_\Delta(N)_{\Z[1/N][\zeta_d]}$.
\end{remark}

Let $X/k$ be a curve over a number field $k$ of genus $g\geq 2$. The \textit{canonical ring}  (or as it is often called, the homogenous coordinate ring) of $X$ is
$$ R(X):=\bigoplus_{d=0}^{\infty}H^{0}(X, \Omega ^{\otimes d}_{X/k}).$$

Let $V:=H^0(X,\Omega_{X/k})$ and $\Sym(V):=\bigoplus_{d=0}^{\infty}\Sym^d(V)$. The identity map $\Sym^1(V)\rightarrow R(X)_1$ (which just sends $V$ to $V$) induces a map of graded rings $f_{can}:\Sym(V)\rightarrow R(X)$. Hence we obtain the \textit{canonical map}
$$X\simeq \Proj (R(X)) \rightarrow \Proj (\Sym (V)) \simeq \P^{g-1}_k.$$
The ideal $I_{can}:=\ker f_{can}\subseteq \Sym (V)$ is called the \textit{canonical ideal}. 

Recall Petri's theorem (\cite{Petri, Saint-Donat}, see also \cite[Section 1.1]{VoightDZB} for a historical overview): for a nonsingular projective curve $X$ of genus $\geq 2$ that is neither hyperelliptic, trigonal (possessing a map $X\rightarrow \PP^1$ of degree 3) nor a smooth plane curve of degree 5, the canonical ring $R(X)$ is generated in degree 1 and $I$ is generated in degree 2. More precisely, we have the following proposition.

\begin{proposition} \label{lem:hyp_crit}
Let $X/k$ be a curve of genus $g \geq 3$. By $V\cdot (I_{can})_2$ we denote the image of $V\otimes (I_{can})_2$ in $(I_{can})_3$.
\begin{itemize}
    \item[a)] $X$ is hyperelliptic over $\overline k$  if and only if $(I_{can})_2 \subseteq \Sym^2 (V)$ is of dimension $\binom{g-1}{2}$.
    \item[b)] Suppose $X$ is not hyperelliptic and not a smooth a plane quintic. Then $X$ is trigonal over $\overline k$ if and only if the dimension of $$(I_{can})_3/(V\cdot (I_{can})_2)$$ is $g-3$.
    \item[c)] Suppose $X$ is a smooth plane quintic over $\overline k$. Then $g=6$ and the dimension of $$(I_{can})_3/(V\cdot (I_{can})_2)$$ is $g-3=3$.
\end{itemize}
\end{proposition}
\begin{proof} This follows directly from \cite[Table Ia]{VoightDZB}.
\end{proof}

\section{Bounds for hyperelliptic and trigonal curves}

\label{sec:hyp_bounds}
\begin{definition}
    We say that a curve is \textit{subhyperelliptic} if it is of gonality $\leq 2$.
\end{definition}

\begin{proposition}[Castelnuovo-Severi inequality, {\cite[Theorem 3.11.3]{Stichtenoth09}}]
\label{tm:CS}
Let $k$ be a prefect field, and let $X,\ Y, \ Z$ be curves over $k$. Let non-constant morphisms $\pi_Y:X\rightarrow Y$ and $\pi_Z:X\rightarrow Z$ over $k$ be given, and let their degrees be $m$ and $n$, respectively. Assume that there is no morphism $X\rightarrow X'$ of degree $>1$ through which both $\pi_Y$ and $\pi_X$ factor. Then the following inequality hold:
\begin{equation}
g(X)\leq m \cdot g(Y)+n\cdot g(Z) +(m-1)(n-1).
\end{equation}
\end{proposition}

Let, as before, $N$ be a integer, $d$ an Atkin-Lehner divisor of $N$ and let $f_{d}:X_0(N)\rightarrow X_0(N)/w_{d}$ be the quotient map by the Atkin-Lehner $w_{d}$ involution and denote by $\nu(d;N)$ the number of complex ramification points of $f_{d}$.

\begin{lemma}[{\cite[p.454.]{Ogg74}}, see also \cite{Kenku77}] \label{lem:ramification_formula} \label{lem:ram}
Let $N > 4$ be an integer. The number of ramification points $\nu(N;N)$ of $f_{N}$ satisfies
$$\nu(N;N)=\begin{cases} h(-4N) +h(-N) & \text{if } N \equiv 3 \pmod 4,\\
h(-4N) & \text{otherwise.}
\end{cases}$$
\end{lemma}

In fact in \cite[Remark 2]{FurumutoHasegawa99} there is also a description of $v(d;N)$ when $d \neq N$. While we used that formula in some explicit calculations, it is not necessary for the main argument.

\begin{proposition} \label{prop:hyp_easy}
Let $X$ be a curve of genus $g$ over an algebraically closed field $k$. Let $f:X\rightarrow \PP^1$ be a morphism of prime degree $p$, and suppose $G  \leq \Aut_{k} X$ is a subgroup such that $X/G$ is of genus 0. If $g > (p-1)(\#G-1)$ then $f$ is cyclic and there is an automorphism $\sigma$ of order $p$ in $G$ such that $f_{\overline k}$ is taking the quotient by $\sigma$.

\end{proposition}
\begin{proof}
Since $g > (p-1)(\#G-1)$, the Castelnuovo-Severi inequality applied to $g:X \to X/G$ and $f$ tells us that $f$ and $g$ have a common factor. But since $f$ is of prime degree, it follows that $g$ factors through $f$. Since, by Galois theory, all intermediate curves $X\rightarrow X' \rightarrow X/G$ are of the form $X'=X/H$ for some subgroup $H$ of $G$, it follows that there exists a subgroup $G'$ of $G$ of order $p$ such that $f$ is quotienting by $G'$, proving the claim.
\end{proof}

\begin{corollary}
Let $p$ be a prime, $N$ an integer, $d$ an Atkin-Lehner divisor of $N$, and $k$ an algebraically closed field of characteristic coprime to $N$. Fix a 'reduction mod $k$' map $\Z[\zeta_d] \to k$. Let $\Delta$ be a subgroup of $ (\Z/N\Z)^\times/\diamondop{-1} \subseteq \Aut_{\Z[1/N]} X_1(N)$, and assume $w_{d,\zeta_d} \Delta (w_{d,\zeta_d})^{-1} = \Delta$ so that $w_{d,\zeta_d}$ induces an automorphism of $X_\Delta(N)_{\Z[1/N][\zeta_d]}$. Let $A$ be the subgroup of $\aut_{\Z[1/N][\zeta_d]} X_\Delta(N)$ generated by $w_{d,\zeta_d}$ and  $ (\Z/N\Z)^\times/\diamondop{-1}/\Delta$.  If $X_0(N)_\C/w_d$ is of genus 0, and $\mathrm{genus}(X_\Delta(N)_\C) > (p-1)(\frac{\phi(N)}{\#\Delta}-1)$, then any morphism $f:X_\Delta(N)_k\rightarrow \PP^1_k$ of degree $p$ is cyclic and there is an automorphism $\sigma \in A_k$ so that $f$ is taking the quotient by $\sigma$.
\end{corollary}

A consequence of the above corollary is that $f$ is the reduction mod $k$ of a map $f_0: X_\Delta(N)_{\Z[1/N][\zeta_d]} \to \P^1_{\Z[1/N][\zeta_d]}$ of degree $p$. Indeed this $f_0$ can be obtained by lifting $\sigma$ to $A$ and then taking the quotient by this lift. In particular this corollary can be used to rule out the existence of maps of degree $p$ over $k$, by showing that there are no maps of degree $p$ over $\C$.
\begin{proof}
Note that the genus of a curve is invariant under good reduction, so the genus conditions over $\C$ imply that $g(X_0(N)_k/w_d) = 0$  and $g(X_\Delta(N))_k > (p-1)(\frac{\phi(N)}{\#\Delta}-1)$. Note that $X_0(N)_k/w_d = X_\Delta(N)_k/A_k$ and $\#A_k = \#A =\frac{\phi(N)}{\#\Delta}$ so that the corollary follows from applying \Cref{prop:hyp_easy} to $A_k$.
\end{proof}

\begin{proposition}\label{prop:conditions}
Suppose $k$ is a field such that $p=\Char k$ does not divide $N$. Let $d>1$ be an Atkin-Lehner divisors of $N$ and $f_d:X_0(N)\rightarrow X_0(N)/w_d$ the quotient map. The following hold:
\begin{enumerate}
    \item if $X_0(N)_k$ is hyperelliptic, then either $(X_0(N)/w_{d})_{\Z[1/N]}$ is of genus $0$ or \\ $\nu(d;N) \leq 4$.
    \item if $f: X_0(N)_k \to \P^1_k$ is a map of odd degree $m$ then $\nu(d;N) \leq 2m$.
    \item if $f: X_0(N)_k \to \P^1_k$ is a map of even degree $m$ then either $\nu(d;N) \leq 2m$ or $f$ factors via $f_{d}$.
\end{enumerate}
\end{proposition}
\begin{proof}
We prove (1); part (2) and (3) are proven analogously, where in (2) we use that an odd degree map cannot factor via $f_{d}$. First observe that since $X_0^+(N)_{\Z[1/N]}$ is smooth, the genus is preserved under base change to $k$. It follows that $X_0(N)/w_{d}$ has genus $0$ over $k$ if and only if it has genus $0$ over $\Z[1/N]$. Hence if its genus is nonzero, then $f_{d}$ is not the hyperelliptic map.

Note that since the degree of a map doesn't change under field extensions the proposition for $k$ follows from the proposition for $\bar k$, so we may assume $k$ is algebraically closed and hence perfect. Applying the Castelnuovo-Severi inequality to the maps $f_{d}$ and the hyperelliptic $h:X_0(N)\rightarrow\PP^1$, we get
\begin{equation}
    g(X_0(N))\leq 2g(X_0(N)/w_{d})+1 \label{eq:1}.
\end{equation}
Applying the Riemann-Hurwitz formula to the map $f_{d}$ we get
\begin{equation}
    2g(X_0(N))=4 g(X_0(N)/w_{d})-2+\nu(d;N) \label{eq:2}.
\end{equation}
Combining \eqref{eq:1} and \eqref{eq:2} yields the claimed result.

\end{proof}

\subsection{Hyperelliptic curves} \label{subsec:hyp}

\begin{definition}
Let $N$ be an integer, $\Delta \leq (\Z /N \Z)^\times/\diamondop{-1}$ and $p$ a prime not dividing $N$. We will say that a pair $(\Delta,p)$ is an exceptional hyperelliptic pair if $X:=X_\Delta(N)$ is not hyperelliptic over $\Z[1/N]$, but $X_{\Fbar_p}$ is hyperelliptic.
\end{definition}

We define
$$S_0:=\left\{ 34,43,45,52,57,64,67,72,73,85,93,97,163,193\right\},$$
$$H:=\left\{ 22, 23, 26, 28, 29, 30, 31, 33, 35, 37, 39, 40, 41, 46, 47, 48, 50, 59, 71 \right\},$$
and
\begin{equation} \label{eq:SH}
SH:=H \cup \{N\leq 32, 36, 49\}.
\end{equation}
The set $H$ is the set of $N$ such the $X_0(N)$ is hyperelliptic and $SH$ is the set of $N$ such that $X_0(N)$ is subhyperelliptic. The set $S_0$ consists of the values of $N$ such that $X_0(N)_{\Z[1/N]}$ is not subhyperelliptic, and $v(d;N)\leq 4$ for all Atkin-Lehner divisors $d$ of $N$.

In the following proposition and at times throughout the remainder of the paper, we will implicitly use the well-known fact that if there exists a morphism $X\rightarrow Y$ over $k$, then $\gon_k X \geq \gon_k Y$ (see for example \cite[Proposition A.1 (vii)]{Poonen2007}).

\begin{proposition} \label{cor:S0}
Let $\Delta=(\Z /N \Z)^\times/\diamondop{-1}$, i.e. $X_\Delta(N)=X_0(N)$. If $(\Delta,p)$ is an exceptional hyperelliptic pair, then $N \in S_0$.
\end{proposition}
\begin{proof} Suppose $X_0(N)_{\F_p}$ is hyperelliptic.
By \Cref{prop:conditions} (1), it follows that either $X_0(N)_{\Z[1/N]}$ is hyperelliptic (in which case $(\Delta,p)$ is not exceptional) or $v(d;N)\leq 4$ for every Atkin-Lehner divisor $d$ of $N$. Using \Cref{lem:ram} we compute that $v(d;N)\leq 4$ for all Atkin-Lehner divisors $d$ of $N$ and $X_0(N)_{\Z[1/N]}$ is not subhyperelliptic only if $N\in S_0 \cup \{88,148,232\}$.

To rule out the values $N\in  \{88,148,232\}$ we note that for each of these $N$, there exists a divisor $n|N$ such that $X_0(n)_{\Z[1/n]}$ is not subhyperelliptic and $v(d;n)> 6$ for some Atkin-Lehner divisor $d$ of $n$. Hence for all primes $p$ not dividing $N$ (and hence not dividing $N$), it follows by the same argument as above that $X_0(n)_{\F_p}$ is not hyperelliptic. Now it follows by \cite[Proposition A.1 (vii)]{Poonen2007} that $X_0(N)_{\F_p}$ is not hyperelliptic.
\end{proof}




\subsection{Trigonal curves} \label{subsec:trig}

\begin{definition}
Let $N$ be an integer, $\Delta \leq (\Z /N \Z)^\times/\diamondop{-1}$ and $p$ a prime not dividing $N$. We will say that a pair $(\Delta,p)$ is an exceptional trigonal pair if $X_\Delta(N)$ is not trigonal over
$\Z[1/N]$, but is trigonal over ${{\Fbar}_p}$. We will say $(N,p)$ is an exceptional pair if $X_0(N)$ is not trigonal over
$\Z[1/N]$, but is trigonal over ${{\Fbar}_p}$.

\end{definition}
Let $S_1$ be the following set
\begin{align*}
S_1:=\{&34, 37, 38, 40, 43, 44, 45, 48, 50, 52, 53, 54, 57, 58, 61, 64, 67, 72, 73, 76, \\
&81, 85, 88, 93, 97, 106, 108, 109, 121, 157, 162, 163, 169, 193, 277, 397\}.
\end{align*}

\begin{proposition} \label{cor:S1}
Let $X=X_0(N)$. If $(N,p)$ is an exceptional trigonal pair, then $N \in S_1$.
\end{proposition}
\begin{proof}
    Suppose $X_0(N)_{{\Fbar}_p}$ is trigonal.
By \Cref{prop:conditions} (2), it follows that either $X_0(N)_{\Z[1/N]}$ is trigonal (in which case $(N,p)$ is not exceptional) or $v(d;N)\leq 6$ for an Atkin-Lehner divisors $d$ of $N$. Using \Cref{lem:ram} we compute that $v(d;N)\leq 6$ and $X_0(N)_{\Z[1/N]}$ is not trigonal only if $N\in S_1\cup\{148, 172, 232, 268, 652\}$.

To rule out the values $N\in \{148, 172, 232, 268, 652\}$ we note that for each of these $N$, there exists a divisor $n|N$ such that $X_0(n)_{\Z[1/n]}$ is of gonality $>3$ and $v(d;n)> 6$ for some Atkin-Lehner divisor $d$ of $n$. Hence for all primes $p$ not dividing $N$ (and hence not dividing $n$), it follows by the same argument as above that $X_0(n)_{\F_p}$ is not trigonal. Now it follows by \cite[Proposition A.1 (vii)]{Poonen2007} that $X_0(N)_{\F_p}$ is not trigonal.
\end{proof}

\section{Checking hyperellipticity and trigonality of a given {$X_\Delta(N)$} over {$\F_p$} for all {$p$}}
\label{sec:hyp_small}
In \Cref{sec:hyp_bounds} we showed that to find all exceptional hyperelliptic pairs $(\Delta,p)$ we need to consider only finitely many subgroups $\Delta$, i.e. only those for which either the level is in $S_0$ or such that $X_0(N)$ is subhyperelliptic. Similarly, to determine all exceptional trigonal pairs $(\Delta,p)$ we need to consider only the finitely many subgroups $\Delta$ for which either the level is in $S_1$ or such that $X_0(N)$ is of gonality $\leq 3$ over $\overline \Q$.

In this section we explain how to find, for a given $\Delta$, all the $p$ such that $(\Delta, p)$ is an exceptional hyperelliptic or trigonal pair.



We will first need the following lemma.

\begin{lemma}Let $R$ be a discrete valuation ring with residue field $k$, fraction field $K$ and uniformizer $\pi$. Let $X$ be a nice (meaning smooth, projective, and geometrically integral) curve over $R$ and suppose $\mathcal L$ is a line bundle on $X$ such that $\mathcal L(X_k)$ and $\mathcal L(X_K)$ have the same dimension, then the map $\mathcal L(X)\otimes_R k \to \mathcal L(X_k)$ is an isomorphism.
\end{lemma}
\begin{proof}
This is done by taking global sections of the exact sequence $$0 \to \mathcal L \xrightarrow{\cdot \pi} \mathcal L \to \mathcal L/\pi \mathcal L \to 0.$$
Since $\mathcal L/\pi \mathcal L \cong \mathcal L \otimes k$, taking global sections of this sequence gives an injection $\mathcal L(X)\otimes_R k \to \mathcal L(X_k)$. Comparing dimensions shows that it has to be an isomorphism.
\end{proof}

Let $X$ be nice curve over $\Z[1/N]$ of genus $g>2$. We use the notation set up in \Cref{sec:background}, in particular, $V:=H^0(X,\Omega_{X_{\Z[1/N]}})$. We have that $V$, $\Sym^2(V)$ and $R(X)_2=H^0(X, \Omega ^{\otimes 2}_{X_{\Z[1/N]}})$ are free $\Z[1/N]$-modules of rank $g$, $\binom{g+1}{2}$ and $3g-3$, respectively.

From the previous lemma we get
\begin{align}
    V\otimes \F_p &\simeq H^0(X_{\F_p}, \Omega_{X_{\F_p}}), \label{eq1.1}\\
    \Sym^m(V) \otimes \F_p &\simeq \Sym^m(V\otimes \F_p), \label{eq1.2}\\
    R(X)_m \otimes \F_p \simeq R(X_{\F_p})_m&=H^0(X_{\F_p}, \Omega ^{\otimes m}_{X_{\F_p}}).  \label{eq1.3}\
\end{align}

The degree $m$ part $(I_{can, \F_p})_m$ of the canonical ideal of $X_{\F_p}$ can be represented as
\begin{equation}(I_{can, \F_p})_m=\ker \left( f^m_{can,\F_p}: \Sym^m(H^0(X_{\F_p}, \Omega_{X_{\F_p}}))\rightarrow H^0(X_{\F_p}, \Omega ^{\otimes m}_{X_{\F_p}}) \right). \label{eq_main}
\end{equation}

\subsection{Hyperelliptic curves}
It follows by \Cref{lem:hyp_crit} and \cite[Table Ia]{VoightDZB} that $X_{{\Fbar}_p}$ is hyperelliptic if and only if
\begin{equation}
    \dim (I_{can, \F_p})_2= \binom{g-1}{2}.
\end{equation}

Now the map $f^2_{can,\F_p}$ is the reduction mod $p$ of the map $f^2_{can}: \Sym^m(V) \to R(X)_2$. By putting a matrix representing the map $f^2_{can}$ into Smith normal form it is easy to see modulo which primes the dimension of the kernel is $\binom{g-1}{2}$. So we have a criterion to detect the primes of hyperelliptic reduction.

To explicitly determine $H^{0}(X_\Gamma, \Omega _{X_\Gamma})$ for a modular curve $X_\Gamma$ corresponding to the congruence group $\Gamma \supset \Gamma(N)$, we use the isomorphism \cite[Eq. 12.1.4]{DiamondIm}
\begin{equation}H^{0}(X_{\Gamma, \Z[1/N]}, \Omega _{X_{\Gamma,\Z[1/N]}}) \cong S_{2}(\Gamma,\Z[1/N]),\label{eq:cuspforms}\end{equation}
where for a ring $R$,we denote by $S_{2}(\Gamma,R)$ is the space of cusp forms of weight $2$ with coefficients in $R$. The map
$$ f^2_{can,\Z[1/N]}: \Sym^2(H^0(X, \Omega_{X_{\Z[1/N]}}))\rightarrow H^0(X_{\Z[1/N]}, \Omega ^{\otimes 2}_{\Z[1/N]}) $$
can be computed on $\mathrm{Sym}^2 S_2(\Gamma,\Z[1/N])$ by multiplying $q$-expansions. On the other hand, $(\Omega^1_{X/\Z[1/N]} \otimes \Omega^1_{X/\Z[1/N]})$ can be identified with the subspace of $S_4(\Gamma,\Z[1/N])$ of cusp forms of weight 4 that have a double zero at all cusps.

In particular, after obtaining a matrix representing the $\Z[1/N]$-module homomorphism $f^2_{can,\Z[1/N]}$ as explained above, and putting it in Smith normal form, one can easily read out the exact primes $p$ where the rank of this matrix will change upon reduction modulo $p$. So we have translated everything into ranks of matrices that can easily be computed in terms of cusp forms.

\subsection{Trigonality of nonhyperelliptic curves} \label{subsec:trig_small}

For the entirety of this section let $X$ be nice curve over $\Z[1/N]$ of genus $g>3$ such that furthermore $X_{{\Fbar}_p}$ is not hyperelliptic for any prime $p$ coprime to $N$. Then we can also detect whether $X_{{\Fbar}_p}$ is either trigonal  or a smooth plane quintic. Namely by \Cref{lem:hyp_crit}  this happens if and only if
\begin{equation}
\dim \left((I_{can,\F_p})_3/(V \otimes \F_p \cdot (I_{can,\F_p})_2)\right) = g-3.
\end{equation}

The assumption that $X_{\F_p}$ is not hyperelliptic for all primes $p$ coprime to $N$ implies that $R(X_{\F_p})_m$ is generated in degree $1$. In particular the map $f^m_{can,\F_p}$ in \eqref{eq_main} is surjective and the ranks of the matrices associated to $f_{can,\Q}^m$ and $f_{can,\F_p}^m$ are the same, from which it also follows that the dimension of $(I_{can, \F_p})_m$ is the rank of $(I_{can})_m$ as a $\Z[1/N]$-module. As a consequence we have
\begin{equation} (I_{can})_m \otimes \F_p \cong (I_{can, \F_p})_m \label{eq_ican_iso}
\end{equation}
The importance of the above isomorphism is that for all primes $p$ one has that the linear map $\mu_{\F_p}: (V \otimes \F_p) \otimes (I_{can,\F_p})_2 \to (I_{can,\F_p})_3$ is just the reduction modulo $p$ of the linear map:
\begin{equation}
    \mu: V \otimes (I_{can})_2 \to (I_{can})_3.
\end{equation}

From the surjectivity of \eqref{eq_main} one can also compute $\dim (I_{can,\F_p})_m$. Indeed
\begin{align}
    \dim (I_{can,\F_p})_m &= \dim \Sym^m(H^0(X_{\F_p}, \Omega_{X_{\F_p}})) - \dim H^0(X_{\F_p}, \Omega ^{\otimes m}_{X_{\F_p}}), \\
    &= \binom{g+m-1}{m} - (2m-1)(g-1).
\end{align}

Putting the above equalities together one has that the primes of trigonal or smooth plane quintic reduction are exactly the primes such that the matrix $\mu$ has rank $\binom{g+3-1}{3} - 5(g-1) - (g-3)$ modulo $p$. Again, these primes can easily be read of from the matrix $\mu$ after one puts $\mu$ into Smith normal form.

However, the matrix $\mu$ has a domain and codomain whose dimension grow as a cubic polynomial in $g$ so computing this matrix and putting it in Smith normal form might become computationally very expensive once $g$ becomes large. And in fact it does become too expensive for some of the curves for which we wanted to compute the primes of trigonal reduction. However if $X(\Z[1/N])$ contain an integral point then the computation can be significantly sped up, as we will describe now.

 From the discussion in the first paragraph of \cite[p.18]{VoightDZB} we directly get the following lemma.
\begin{lemma}[{\cite[Section 2.5]{VoightDZB}}] \label{VDZB:crit} Let $X_k$ be a non-hyperelliptic curve of genus $\geq 4$ over an algebraically closed field $k$ and $X_{k,2} $ be the variety cut out by $(I_{can,k})_2$, i.e. the quadrics vanishing on $X_k$. Then \begin{itemize}
    \item[(1)] $X_k$ is trigonal or a smooth plane quintic if and only if $X_{k,2}$ is a surface.

    \item[(2)] $X_k$ is not trigonal or a smooth plane quintic if and only if $X_{k,2} =X_k$.
\end{itemize}
\end{lemma}

Now let $P \in X(\Z[1/N])$ be a point. From \Cref{VDZB:crit} it follows that the tangent space $T_P X_{k,2}$ is $1$-dimensional if and only if $X$ is not trigonal or a smooth plane quintic. Since $X$ is not hyperelliptic modulo any prime, the canonical embedding allows one to see $X$ as a subvariety of $\P^{g-1}_{\Z[1/N]}$. Let $x_0, x_1,\ldots,x_{g-1}$ be the coordinates on $\P^{g-1}_{\Z[1/N]}$ of some affine neighborhood of $P$ and $f_i$ be generators of $(I_{can})_2$ on this neighborhood. Because of \eqref{eq_ican_iso} the reductions $f_i$ modulo $p$ are also generators of $(I_{can,\F_p})_2$; we compute $T_{P_{\F_p}} X_{\F_p,2}$ as the kernel of the Jacobian matrix $J=\left(\frac{\delta f_i(P)}{\delta x_i} \right)_{i,j}$ modulo p. This is a matrix that can itself be written down over $\Z[1/N]$, and as before putting it in Smith normal form allows us to easily read out the possible trigonal or smooth plane quintic primes, by computing the primes such that this matrix has rank $< g-2$.

\subsection{Smooth plane quintics}\label{subsec:spq}

We note that the methods of our paper do not distinguish between smooth plane quintics and trigonal curves, hence the methods of \Cref{subsec:trig_small} are also used to detect possible smooth plane quintics. However, since plane smooth quintics necessarily have genus 6, this greatly reduces the number of curves that need to be considered. It fortunately turns out that no intermediate modular curves of genus 6 satisfy the necessary and sufficient conditions of being either trigonal or a smooth plane quintic.

\section{Proof of Theorems \ref{t1} \ref{t2} and \ref{t3}}
\begin{proof}[Proof of \Cref{t1}]

We first prove that there are no exceptional pairs $(\Delta,p)$, where $\Delta$ is the trivial group; this is done as explained in \Cref{sec:hyp_small}. By \Cref{cor:S0} the values that need to be checked are $N\in S_0$.  We obtain that for $p\nmid N$, the curve $X_0(N)_{\F_p}$ is hyperelliptic if and only if $X_0(N)_{\Z[1/N]}$ is.

It remains to consider the nontrivial subgroups $\Delta$. Suppose $\Delta$ is nontrivial and $(\Delta, p)$ is an exceptional hyperelliptic pair. If there exists a morphism of curves $X_k\rightarrow Y_k$ defined over a field $k$ and $X_k$ is hyperelliptic, then it follows that $Y_k$ has to be subhyperelliptic \cite[Proposition A.1. (vii)]{Poonen2007}. It follows, that $X_\Delta(N)$ can be hyperelliptic over $\F_p$ only if $X_0(N)_{\Z[1/N]}$ is hyperelliptic. Since for any $N$ there are finitely many $\Delta$ of level $N$, we are reduced to checking the hyperellipticity of finitely many $X_\Delta(N)$ to complete the proof of \Cref{t1}.

All the computations to verify this take 116 seconds, and we find the unique exceptional pair $\Delta:=\langle 4 \rangle \leq (\Z/ 37\Z)^\times$ and $p=2$.
\end{proof}

\begin{proof}[Proof of \Cref{t2}]
We first determine the exceptional trigonal pairs for the trivial groups $\Delta$, i.e. $X_\Delta(N)=X_0(N)$. By \Cref{cor:S1}, the values that need to be checked are $N\in S_1.$ We follow the procedure described in \Cref{subsec:trig_small} and get the exceptional trigonal pair $X_0(73)$ and $p=2$.

Using the arguments as in the proof of of \Cref{t1}, it follows that it remains to consider nontrivial subgroups $\Delta$, and the ones that need to be considered are the ones such that $X_0(N)_{\Z[1/N]}$ is of gonality $\leq 3$, and in addition the $\Delta$ of level 73 for $p=2$. We check this, and after 23 minutes of computation obtain that there are no additional exceptional pairs.
\end{proof}

\begin{proof}[Proof of \Cref{t3}]
Using the arguments as in the proof of \Cref{t2} (and the fact that we need only consider curves of genus 6), it turns out there are no intermediate modular curves that are smooth plane quintics in any characteristic.
\end{proof}

\section{Fields of definition of trigonal maps on non hyperelliptic curves} \label{sec:fields_of_definition}
While \Cref{t2} completely solves \Cref{q2} B) for intermediate modular curves and $d=3$, it remains to consider \Cref{q2} A) for trigonal curves over finite fields $\F_q$. Let $X/k$ be a curve of genus $g$ with $X(k)\neq \emptyset$ that is trigonal over $\overline{k}$ and not subhyperelliptic. Then $X$ is trigonal over $k$ if $g=3$ (\cite[Proposition A.1 (iv)]{Poonen2007}) or $g\geq 5$ \cite[Corollary 4.6 (i)]{NajmanOrlić}. Hence the only case of interest here is $g=4$.
\begin{proposition} Let $X$ be a curve of genus $4$ over $\Z[1/N]$ of non-hyperelliptic reduction at all primes of good reduction with $X(\Z[1/N])\neq \emptyset$. Then the canonical model of $X$ is a smooth complete intersection of a cubic and a quadric $Q$ in $\PP^3_{\Z[1/N]}$. Let $D$ be the discriminant of an extension of this quadric $Q$ to $\Z$; we assume that $D$ is not a perfect square and let $\O(D)$ be the order of discriminant $D$. Then
\begin{enumerate}
    \item $X$ is trigonal over the quadratic field $\O(D) \otimes \mathbb Q$, but not over $\mathbb Q.$
    \item For all positive odd integers $n$ and all rational primes $p$ coprime to $N$ at which the order $\O(D)$ is maximal, $X$ is trigonal over $\F_{p^n}$ if and only if $p$ splits or ramifies in $\O(D)$.
    \item For all positive even integers $n$ and all rational primes $p$ coprime to $N$, $X$ is trigonal over $\F_{p^n}$.
\end{enumerate}
\end{proposition}
\begin{proof}
Let $k$ be a field of characteristic coprime to $N$ with $X:=X_k$ smooth and non-hyperelliptic. Furthermore, let  $W_{d}^r(X_k) \subset \Pic^d X_k$ denote the Brill-Noether variety corresponding to line bundles that have at least $r+1$ linearly independent global sections. Note that since $X(\Z[1/N]) \neq \emptyset$ we also have $X(k) \neq \emptyset$ and hence any point in $\Pic^{d} X (k)$ actually comes from a $k$-rational line bundle of degree $d$ \footnote{On the other hand, if $X(k) = \emptyset$ and $k$ is perfect, then a point in $\Pic^{d} X (k)$ only gives a line bundle over $\overline k$ that is isomorphic to all its Galois conjugates.}. In particular, every $k$-rational point on the Brill-Noether variety $W_3^1(X_k)$ comes from a line bundle defined over $k$ of degree $3$ with two linearly independent global sections, and hence gives rise to a non-constant $k$-rational function degree $\leq 3$, which is actually of degree 3 by the non-hyperellipticity assumption. Conversely, every $k$-rational function of degree $3$ gives a rational point on $W_3^1(X_k)$. In conclusion, $X_k$ is trigonal if and only if $W_{d}^r(X_k)(k) \neq \emptyset$.

By \cite[p.206]{AOC1}\footnote{We note that the setting of \cite{AOC1} is over the complex numbers, but that the theory carries over to any algebraically closed field.}, the Brill-Noether variety $W_3^1(X_k)$ either has 1 or 2 points over $\overline{k}$. The case $W_3^1(X)$ being isomorphic to $X$ mentioned there cannot happen since we assumed $X$ to not be hyperelliptic. If $\#W_3^1(X_k)(\overline k)=1$ then $\#W_3^1(X_k)(k)=1$ and $X_k$ is trigonal by the above discussion. On the other hand if $\#W_3^1(X_k)(\overline k)=2$, then there are exactly two possibilities. Namely:
\begin{itemize}
    \item[(a)] $\#W_3^1(X_k)(k)=2$ and hence $X_k$ is trigonal,
    \item[(b)] $\#W_3^1(X_k)(k)=0$ and hence $X_k$ is not trigonal but becomes trigonal over the quadratic extension of $k$ over which the two points of $\#W_3^1(X)$ are defined.
\end{itemize}
We let $F_{X_k}$ denote the smallest field extension of $k$ for which $\#W_3^1(X_k)(F_{X_k}) > 0$. So either $F_{X_k} = k$ or it is the quadratic field over which the two points of $\#W_3^1(X)$ are defined in case (b) above.

Since $X_k$ is not hyperelliptic, the canonical model of $X$ is a smooth complete intersection of a cubic and a quadric $Q_k$ in $\P^3_k$. Each $g_3^1$ corresponding to one of the points in $W_3^1(X)(k)$ is a family of lines in $\P^3_k$ intersecting $X_k$ three times, counting multiplicities. Let $L$ denote one of these lines. By Bezout's theorem this line $L$ lies on the quadric $Q_k$. In particular, if $Q_k$ is nonsingular then the lines of the $g_3^1$ actually form a ruling of $Q$ (see \cite[p.206]{AOC1} again). In the nonsingular case, the field $F_{X_k}$  defined above is actually the field of definition of these rulings. If $\Char k>2$, then this field is obtained by adjoining to $k$ the square root of the discriminant of the polynomial defining $Q$ (see e.g \cite[p.45-46]{ClarkJNTB06}).

Part (1) immediately follows from this discussion. Indeed, $F_{X_\Q} = \Q(\sqrt{D}) = \O(D) \otimes \mathbb Q$. So for a field extension $K$ of $\Q$ we have $\# W_3^1(X_\Q)(K) > 0$ if and only if $K$ contains $\Q(\sqrt{D})$. Both of these equivalent conditions are equivalent to $X_K$ being trigonal.

Part $(3)$ follows similarly since $F_{X_{\F_p}}$ is either $\F_p$ or $\F_{p^2}$. So in particular, $\F_{p^n}$ contains $F_{X_{\F_p}}$ if $n$ is even.

Now we prove part $(2)$ in the case where $p$ splits or ramifies in $\O(D)$. Let $\mathfrak{p}$ be a prime of $\O(D)$ lying over $p$. Then by the maximality assumption on $\O(D)$ at $p$, the ring $\O(D)_{\mathfrak{p}}$ is a discrete valuation ring and the residue field of $\O(D)_{\mathfrak{p}}$ is $\F_p$. By part (1) we have $\gon X_{\O(D) \otimes \Q}  = 3$. Since the gonality of a curve can only decrease under specialization and field extensions we have $$\gon X_{\F_{p^n}} \leq \gon X_{\F_p} \leq \gon X_{\O(D)_{\mathfrak{p}} \otimes \Q} \leq \gon X_{\O(D) \otimes \Q} =3.$$ Hence $\gon X_{\F_{p^n}}=3$ by the non-hyperellipticity assumption.

What remains is to show part $(2)$ when $p$ is inert. Assume for the moment $p>2$. Since $p$ is inert in $\O(D)$ we know that $F_{X_{\F_p}} = \F_p(\sqrt{D})$ is a quadratic extension of $\F_p$ and hence isomorphic to $\F_p^2$. It follows that $X$ is trigonal over $\F_p^n$ if and only if $\F_p^n$ contains $\F_p^2$, which happens exactly when $n$ is even. So if $n$ is odd, $X_{\F_p^n}$ is not trigonal.

The case $p=2$ is slightly more subtle. While it might be possible to deal with this case abstractly, we found it clearer to take a more explicit approach. Let's write the quadric $Q$ in $\P^3$ as
$$Q = \sum_{i=0}^3 \sum_{j=0}^i a_{i,j}x_ix_j.$$ For this quadric we will also use the matrix form notation
$$
Q = \begin{bmatrix}
a_{0,0} & a_{0,1} & a_{0,2} & a_{0,3} \\
*       & a_{1,1} & a_{1,2} & a_{1,3} \\
*       & *       & a_{2,2} & a_{2,3} \\
*       & *       & *       & a_{3,3}
\end{bmatrix}.$$
Let $P \in X(\Z[1/N]) \subseteq \P^3(\Z[1/N])$ be a point. By choosing a suitable set of coordinates on $\P^3$ we may assume $P = (0:0:0:1)$ and hence $a_{3,3}=0$. By a further change of coordinates we may assume $a_{1,3}=a_{2,3}=0$ as well, so that $Q$ looks like:
$$
\begin{bmatrix}
a_{0,0} & a_{0,1} & a_{0,2} & a_{0,3} \\
*       & a_{1,1} & a_{1,2} & 0 \\
*       & *       & a_{2,2} & 0 \\
*       & *       & *       & 0
\end{bmatrix}.$$

With $Q$ as above, we have $D = a_{0,3}^2(a_{1,2}^2 - 4a_{1,1}a_{2,2})$. The assumption for $p$ to be inert in $\O(D)$ is equivalent to $D \equiv 5 \mod 8.$ In particular $Q_{\F_p}$ is nonsingular and  $a_{0,3} \equiv a_{1,2} \equiv a_{1,1} \equiv a_{2,2} \equiv 1 \mod 2$.

Since $Q_{\F_2}$ is nonsingular it has two rulings. These two rulings both contain a line passing through $P$. Additionally, starting from $P$ one can let $T \subset P^3_{\F_2}$ be the tangent space to $Q_{\F_2}$ at $P_{\F_2}$, and $T \cap Q_{\F_2}$ will be exactly the union of these same two lines. In particular if these two lines are interchanged by the action of Galois, then these two rulings will be interchanged as well.

Now let's compute $T$ on the affine chart where $x_3 \neq 0$ and let $X_0,X_1,X_2$ be the affine coordinates on $\A^3_{\F_2}$. Then the tangent space $T$ is given by $X_0=0$, so that the union of two lines on $Q_{\F_2}$ passing through $P_{\F_2}$ can be described by $X_0=0$ and  $a_{1,1}X_1^2 + a_{1,2}X_1X_2 +a_{2,2}X_2^2$. Since $a_{1,1} \equiv a_{2,2} \equiv a_{1,2} = 1 \pmod 2$, the polynomial $a_{1,1}X_1^2 + a_{1,2}X_1X_2 +a_{2,2}X_2^2$, is irreducible and hence so is also the scheme $T \cap Q_{\F_2}$. This can only happen if the geometric lines generating $T \cap Q_{\F_2}$ are Galois conjugates. In particular, the two rulings of $Q_{\F_2}$ are swapped by the action of Galois, $W_3^1(X_{\F_2})$ consist of two points whose field of definition is $\F_{2^2}$ and hence $X_{\F_2^n}$ is not trigonal when $n$ is odd.



\end{proof}

\begin{remark}
  In case $X_k$ is of genus $4$ with $X_k(k)\neq \emptyset$ and is neither hyperelliptic or trigonal over $k$, then it is necessarily tetragonal over $k$, see e.g. \cite[Proposition A.1 (iv)]{Poonen2007}.
\end{remark}

The table below lists all genus $4$ intermediate modular curves by specifying their level $N$, the group $\Delta(N)$, and the disciriminant $D$ of the corresponding order $\O$. Note that $\Delta(N)=(\Z/N\Z)^*/\pm 1$ means $X_\Delta(N)=X_0(N)$ and $D$ a perfect square means that the curve is trigonal over $\Q$. We note that the cases $X_\Delta(N)=X_0(N)$ had already been previously solved in \cite[p. 136]{HasegawaShimura_trig}.

\begin{table}[H]
\centering
\begin{tabular}{r|r|r}
$N$ & $\Delta(N)$ & $D$\\
\hline
25 & $\diamondop{7}$ & 5 \\
26 & $\diamondop{17}$ & 1 \\
26 & $\diamondop{5}$ & 1 \\
28 & $\diamondop{3}$ & 0 \\
28 & $\diamondop{13,15}$ & 4 \\
29 & $\diamondop{4}$ & 1 \\
37 & $\diamondop{4}$ & 0 \\
37 & $\diamondop{8}$ & 1 \\
38 & $(\Z/38\Z)^\times/\diamondop{-1}$ & -3 \\
44 & $(\Z/44\Z)^\times/\diamondop{-1}$ & -8 \\
50 & $\diamondop{19}$ & 1 \\
53 & $(\Z/53\Z)^\times/\diamondop{-1}$ & -15 \\
54 & $(\Z/54\Z)^\times/\diamondop{-1}$ & 1 \\
61 & $(\Z/61\Z)^\times/\diamondop{-1}$ & -4 \\
81 & $(\Z/81\Z)^\times/\diamondop{-1}$ & 0 \\
\end{tabular}
\caption{Field of definition of trigonal maps}
\label{tab:trigonal_fd}
\end{table}


\appendix

\section{}
\label{appendix}
\subsection{Discriminants of class number at most 100}\label{sec:small_class_numbers}
Let $D<0$ be a fundamental discriminant and $f$ a positive integer. Then the following formula relates class numbers of (not necessarily fundamental) discriminants to those of fundamental discriminants (see \cite[Corollary 7.28]{Cox})
\begin{equation}
h(Df^2) = \frac {h(D)f} {u_{D,f}} \prod_{p | f} \left( 1- \left(\frac D p\right) \frac 1 p \right),
\end{equation}
where $w_{D,f} = 3$ if $D=-3$ and $f \neq 1$,  $w_{D,f} = 2$ if $D=-4$ and $f \neq 1$ and $w_{D,f}=1$ otherwise. Using this formula and the list of negative fundamental discriminants of class number $\leq 100$ from  \cite[Table 4.]{Watkins2004}, it is straightforward to compile a list of all negative discriminants of class number $\leq 100$. For each class number $h \leq 100$ we record the number of discriminants $Df^2 < 0$ such that $h(Df^2)=h$ in the second column, and the smallest $Df^2$ such that $h(Df^2)=h$ in the third column.
\begin{table}[H]
\centering
\begin{tabular}{rrr|rrr|rrr|rrr}
  h &   \# & smallest &   h &   \# & smallest &   h &   \# & smallest &   h &   \# & smallest  \\
  \hline
  1 &   13 &     -163 &  26 &  227 &  -103027 &  51 &  217 &  -546067 &  76 & 1381 & -1086187 \\
  2 &   29 &     -427 &  27 &  136 &  -103387 &  52 & 1003 &  -457867 &  77 &  236 & -1242763 \\
  3 &   25 &     -907 &  28 &  623 &  -126043 &  53 &  130 &  -425107 &  78 &  921 & -1004347 \\
  4 &   84 &    -1555 &  29 &   94 &  -166147 &  54 &  806 &  -532123 &  79 &  200 & -1333963 \\
  5 &   29 &    -2683 &  30 &  473 &  -137083 &  55 &  177 &  -452083 &  80 & 3851 & -1165483 \\
  6 &  101 &    -4075 &  31 &   83 &  -133387 &  56 & 1809 &  -494323 &  81 &  338 & -1030723 \\
  7 &   38 &    -5923 &  32 & 1231 &  -164803 &  57 &  237 &  -615883 &  82 &  486 & -1446547 \\
  8 &  208 &    -7987 &  33 &  158 &  -222643 &  58 &  360 &  -586987 &  83 &  174 & -1074907 \\
  9 &   55 &   -10627 &  34 &  260 &  -189883 &  59 &  144 &  -474307 &  84 & 2990 & -1225387 \\
 10 &  123 &   -13843 &  35 &  111 &  -210907 &  60 & 2352 &  -662803 &  85 &  246 & -1285747 \\
 11 &   46 &   -15667 &  36 & 1303 &  -217627 &  61 &  149 &  -606643 &  86 &  553 & -1534723 \\
 12 &  379 &   -19723 &  37 &   96 &  -158923 &  62 &  386 &  -647707 &  87 &  313 & -1261747 \\
 13 &   43 &   -20563 &  38 &  283 &  -289963 &  63 &  311 &  -991027 &  88 & 2769 & -1265587 \\
 14 &  134 &   -30067 &  39 &  162 &  -253507 &  64 & 2915 &  -693067 &  89 &  206 & -1429387 \\
 15 &   95 &   -34483 &  40 & 1418 &  -274003 &  65 &  192 &  -703123 &  90 & 1508 & -1548523 \\
 16 &  531 &   -35275 &  41 &  125 &  -296587 &  66 &  856 &  -958483 &  91 &  249 & -1391083 \\
 17 &   50 &   -37123 &  42 &  595 &  -301387 &  67 &  145 &  -652723 &  92 & 1590 & -1452067 \\
 18 &  291 &   -48427 &  43 &  123 &  -300787 &  68 & 1227 &  -819163 &  93 &  354 & -1475203 \\
 19 &   59 &   -38707 &  44 &  909 &  -319867 &  69 &  292 &  -888427 &  94 &  598 & -1587763 \\
 20 &  502 &   -58843 &  45 &  231 &  -308323 &  70 &  702 &  -821683 &  95 &  273 & -1659067 \\
 21 &  118 &   -61483 &  46 &  328 &  -462883 &  71 &  176 &  -909547 &  96 & 7265 & -1684027 \\
 22 &  184 &   -85507 &  47 &  117 &  -375523 &  72 & 4046 &  -947923 &  97 &  208 & -1842523 \\
 23 &   78 &   -90787 &  48 & 2893 &  -335203 &  73 &  137 &  -886867 &  98 &  707 & -2383747 \\
 24 & 1042 &  -111763 &  49 &  146 &  -393187 &  74 &  472 &  -951043 &  99 &  396 & -1480627 \\
 25 &  101 &   -93307 &  50 &  440 &  -389467 &  75 &  353 &  -916507 & 100 & 2304 & -1856563 \\
\end{tabular}
\caption{Imaginary class numbers smaller than 100}
\label{tab:small_class_numbers}
\end{table}

The full list of all discriminants of class number $\leq 100$ can be found at \cite{code} in the \verb|small_class_numbers_and_ramification.ipynb| file.

\newpage
\subsection{Ramification degrees of \texorpdfstring{$X_0(N) \to X_0(N)^+$}{X0(N) -> X0(N)+} at most 100}

Using \Cref{lem:ramification_formula} and \Cref{tab:small_class_numbers}
one can easily compile a list of all integers $N$ such that the ramification degree of $X_0(N) \to X_0(N)^+$ is at most 100. For each degree $d \leq 100$ we record the number of integers $N$ that the ramification degree of $X_0(N) \to X_0(N)^+$ equals $d$, as well at the maximum of these $N$. Note that by the Riemann-Hurwitz formula the ramification degree is always even.
\begin{table}[H]
\centering
\begin{tabular}{rrr|rrr|rrr|rrr}
  d &   \# & largest &   d &   \# & largest &   d &   \# & largest &   d &   \# & largest \\
  \hline
  2 &   18 &      58 &  28 &  257 &   13297 &  54 &  125 &   48742 &  80 & 2309 &  120712 \\
  4 &   48 &     253 &  30 &   88 &   14422 &  56 &  992 &   62302 &  82 &  118 &  151237 \\
  6 &   32 &     652 &  32 &  790 &   18748 &  58 &   77 &   48778 &  84 & 1019 &  166798 \\
  8 &  128 &    1012 &  34 &   56 &   18397 &  60 &  817 &   83218 &  86 &  106 &  137197 \\
 10 &   39 &    1318 &  36 &  482 &   22768 &  62 &   96 &   85402 &  88 & 1413 &  150382 \\
 12 &  173 &    2608 &  38 &   74 &   30493 &  64 & 1857 &  106177 &  90 &  238 &  149053 \\
 14 &   35 &    2293 &  40 &  785 &   30178 &  66 &  175 &   92698 &  92 &  577 &  189352 \\
 16 &  329 &    4048 &  42 &  130 &   29437 &  68 &  493 &  102958 &  94 &  112 &  184438 \\
 18 &   62 &    5692 &  44 &  375 &   34318 &  70 &  127 &   94378 &  96 & 4289 &  198958 \\
 20 &  225 &    5377 &  46 &   78 &   47338 &  72 & 1963 &  134773 &  98 &  132 &  161302 \\
 22 &   40 &    6637 &  48 & 1618 &   41728 &  74 &  104 &   91228 & 100 &  842 &  200722 \\
 24 &  576 &   10432 &  50 &   75 &   43717 &  76 &  506 &  121972 &     &      &         \\
 26 &   63 &   11302 &  52 &  389 &   50317 &  78 &  170 &   92458 &     &      &         \\
\end{tabular}
\caption{Ramification degree smaller than 100}
\label{tab:small_ramification_degrees}
\end{table}

The full list of all integers $N$ of ramification degree $\leq 100$ can be found at \cite{code} in the \verb|small_class_numbers_and_ramification.ipynb| file.

\bibliographystyle{siam}
\bibliography{bibliography1}

\end{document}